\documentclass[12pt,reqno]{amsart}
\usepackage{enumerate, latexsym, amsmath, amsfonts, amssymb, amsthm, graphicx, color}
\def\pmod #1{\ ({\rm{mod}}\ #1)}
\def\Z{\Bbb Z}

\def\M{\Bbb M}

\def\bg{\bigg}
\def\({\bg(}
\def\){\bg)}
\def\t{\text}

\def\eq{\equiv}

\def\T{\overline}
\def\B{\tilde}
\def\M{\pmod}

\theoremstyle{plain}
\newtheorem{theorem}{Theorem}

\newtheorem{problem}{Problem}
\newtheorem{lemma}{Lemma}
\newtheorem{corollary}{Corollary}

\newtheorem{definition}{Definition}

\theoremstyle{remark}

\makeatletter
\@namedef{subjclassname@2020}{%
  \textup{2020} Mathematics Subject Classification}
\makeatother
 \vspace{4mm}

\begin{document}

\title[Balanced paring of $\{1,2,\ldots,(p-1)/2\}$ for $p\eq 1 \pmod{4}$]{Balanced paring of $\{1,2,\ldots,(p-1)/2\}$ for $p\eq 1 \pmod{4}$}

\
\author[Chao Huang]{Chao Huang}
\address{(Chao Huang) Department of Mathematics\\ Nanjing University\\
Nanjing 210093, People's Republic of China}

\email{DG1921004@smail.nju.edu.cn}

\date{}

\begin{abstract}
Let $p\eq1 \pmod{4}$ be a prime. Write $t = \prod_{x=1}^{(p-1)/2}x$. Since $t ^2\eq -1 \pmod{p}$ , we can divide $\{1,2,\ldots,(p-1)/2\}$ into $(p-1)/4$ ordered pairs so that each pair, say $<a,\B{a}>$ , satisfies that $t a \eq \pm \B{a} \pmod{p}.$ For any two such pairs, assume $a<\B{a}, b<\B{b}, a<b $, then there are three possibilities for their relative order : $a<\B{a} < b< \B{b}$ , $a< b < \B{a} < \B{b}$ , $a< b < \B{b}< \B{a}$. We show this paring is balanced in the sense that the three cases occur with equal frequencies. Utilizing properties of this paring we solve one problem raised by Zhi-Wei Sun concerning the sign of permutation related to quadratic residues.
\end{abstract}



\maketitle

\section{Introduction}
\setcounter{lemma}{0}
\setcounter{theorem}{0}
\setcounter{corollary}{0}
\setcounter{remark}{0}
\setcounter{equation}{0}
\setcounter{conjecture}{0}

Let p be an odd prime and $\{ \}_p $ denote the least nonnegative residue modulo p. For any integer x we define
$$\T{x} =
\begin{cases} p- \{ x\}_p   &\t{when}\ \  \{ x\}_p> p/2 ,
\\ \{ x\}_p &\t{when}\ \ \{ x\}_p< p/2.
\end{cases} \\$$

 Thus $ \T{x}=\T{-x}=\T{p\pm x}$. Note that addition under bar doesn't work well. For example, when $p=13$, we have $\T{\T{4}+\T{8}}= \T{4+5}=4\neq 1= \T{4+8}.$ However, if we define $x\ast y = \T{xy}$ then $\T{x}\ast \T{y} = x\ast y$ for any $x,y\in Z$ and $\{1,2,\ldots,(p-1)/2\}$ is made into a multiplicative group. It can be viewed as quotient group of multiplicative group $ (\Z / p \Z)^*$ by $\{1, -1\}.$ Hence it's cyclic and generated by any primitive root g of p in it.

 We'd like to call $\{1,2,\ldots,(p-1)/2\}$ the group of positive residues as Conway refer those $x$ with $\{x\}_p <p/2$ as 'positive modulo p'.(\cite[p. 128]{C})

Our work is motivated by one of Zhi-Wei Sun's conjectures.(\cite[p. 280]{S19}). The case for $p\eq 3 \pmod{4}$ is easy as we'll show in Section 4.
\begin{problem}
Let p be an odd prime and $m\in Z $ with $p\nmid m.$ Let $s'(p)$ be the number of pairs $(i,j)$ such that $1\leq i < j \leq (p-1)/2$ and  $\T{ mi^2} > \T{mj^2}.$ Prove that $s'(p) \eq \lfloor \frac{p+1}{8} \rfloor  \pmod{2}$, where $ \lfloor \ \rfloor$ is the floor function.
\end{problem}

When $p\eq 1 \pmod{4}$ we find that certain combinatoric structure of $\{1,2,\ldots,(p-1)/2\}$ is essential. Write $t = \prod_{x=1}^{(p-1)/2}x$. It's well known by Wilson's theorem $t^2\eq -1 \pmod{p}$. If we define a permutation $\tau$ of this group of positive residues by $\tau(x) = \T{tx}$, then $\tau(\tau(x))= \T{t^2x}=\T{-x}=x$ for $ x \in \{1,2,\ldots,(p-1)/2\}$. Then we divide the group of positive residues into $(p-1)/4$ ordered pairs as follows.
\begin{definition}
Let $p\eq 1 \pmod{4}$ be a prime and $t = \prod_{x=1}^{(p-1)/2}x$. Define $V = \{ x | 1\leq x \leq (p-1), \  x < \T{tx} \},$
and for each $x\in V$ ,let $\B{x}= \T{tx}$.

We call the set of ordered pairs $\{ <x,\B{x}> | x\in V \}$ the paring of $ \{1,2,\ldots,(p-1)/2\}$, or simply the paring for $(p-1)/2.$
\end{definition}

 We always let $a,b$ denote elements of V and refer to $\B{a}$ as partner of $a$ and vice versa. Here is an example for $ p=29 , t=12.$
\renewcommand\arraystretch{1.2}
\begin{center}
\begin{tabular}{|c|c|c|c|c|c|c|c|}\hline
\ $a$ & 1 & 2 & 3& 4& 6 & 8 & 11  \\ \hline
\ $\B{a} $ & 12 & 5 & 7& 10 & 14 & 9& 13   \\  \hline
\end{tabular}
\end{center}

The reason to consider such a paring is that $ \{a^2\}_p + \{\B{a}^2\}_p = p$ and $ \T{a^2} = \T{\B{a}^2}$ since $ \B{a}^2 \eq (ta)^2 \eq -a^2 \pmod{p}$ so that it's reasonable to put $a$ and $\B{a}$ together in solving our problem. This pairing turns out to have elegant properties. For example we may form an interesting table.
\renewcommand\arraystretch{1.2}
\begin{center}
\begin{tabular}{|c|c|c|c|c|c|c|c|}\hline
\ $\B{a}-a$ & 11 & 3 & 4 &  6& 7 & 1 & 2  \\ \hline
\ $\T{\B{a}+a} $ & 13 & 7 & 10& 14 & 9 & 12 & 5    \\ \hline
\end{tabular}
\end{center}

Compare the two tables we immediately make an observation that this table can be obtained from former one by permutation of columns. Moreover, for two such pairs, there are three possibilities for their relative order as the figure shows :
$$[\bigcirc, \triangle,\bigcirc,\triangle], \  [\bigcirc, \bigcirc,\triangle,\triangle], \   [\bigcirc, \triangle,\triangle,\bigcirc].$$
Since by definition $a< \B{a}$ for each $a \in V$, we define
$$\begin{cases} X = \{(a,b) \in V \times V | \ a< b < \B{a} < \B{b} \},   \\
   Y = \{(a,b) \in V \times V | \  a< \B{a} < b < \B{b} \},   \\
    Z = \{(a,b) \in V \times V | \  a< b < \B{b} < \B{a} \}.  \\
\end{cases}$$

In Section 2 we first show that $ |X| \eq s'(p) \pmod{2} $ and solve Problem 1 utilizing above observation. Since this proof
involves Jacobi Sums of quartic character, we try to find a more elementary proof. This leads to the amazing discovery that the pairing is balanced in the sense that $|X| = |Y| = |Y|$, which enables us to give a second combinatoric proof in Section 3.

\section{}

 From now on we always assume $p \eq 1 \pmod{4}$ be a prime number except in the last section. Recall that $t = \prod_{x=1}^{(p-1)/2}x$ so that $t^2 \eq -1 \pmod{p}$. And for each pair $< a, \B{a}>$, we have $ a < \B{a}, a= \T{t\B{a}}, \B{a}= \T{ta}.$

 First we show that the sets $X, Y , Z $ defined in the last section is essential in solving Problem 1.
 \begin{theorem}
 Let $p \eq 1 \pmod{4} $ be an odd prime and $m\in Z $ with $p\nmid m.$ Let $s'(p)$ be the number of pairs $(i,j)$ such that $1\leq i < j \leq (p-1)/2$ and
 $\T{ mi^2} > \T{mj^2}.$ Define $X = \{(a,b) \in V \times V |  \ a< b < \B{a} < \B{b} \}$. Then $$ s'(p) \eq |X| \pmod{2}.  $$

 \end{theorem}
 \begin{proof}
 For any $a \in V$, pair $<a,\B{a}>$ is not counted by $s'(p)$ since $\T{ma^2} =\T{m\B{a}^2}$

 Thus for any pair $(i,j)$ counted by $s'(p)$ we have $i,j\in \{a,\B{a},b,\B{b}\}$ for different $a,b \in V.$ Assume $a < b.$
 We list all the possibility for $(i,j)$ to be inverse under two assumptions. For example , the item 'None' means when $(a,b) \in Y $ and $\T{ ma^2} <
 \T{mb^2}$, any pair from $ \{a,\B{a},b,\B{b}\}$ is not counted by $s'(p)$.

 \renewcommand\arraystretch{1.3}
 \begin{tabular}{|c|c|c|}\hline
 (i,j) & \ when $\T{ ma^2} > \T{mb^2}$  &  when $\T{ ma^2} < \T{mb^2}$ \\  \hline
 \ X:\ $a< b < \B{a} < \B{b}$ & $(a,b), (a,\B{b}),(\B{a},\B{b})$    & $ (b,\B{a})$  \\ \hline
 \ Y:\ $a< \B{a} < b < \B{b}$ &  $(a,b), (a,\B{b}),(\B{a},\B{b})$ ,$ (\B{a},b) $  &  None    \\ \hline
 \ Z:\ $a< b < \B{b} < \B{a}$ &  $(a,b),(a,\B{b})$ & $ (b,\B{a}),(\B{b},\B{a})$ \\ \hline
 \end{tabular} \\

 From the table we notice that only for those $(a,b) \in X$ , there are odd number of inverse pairs. Hence the conclusion easily follows. \\
 \end{proof}

Thus it suffices to determine the parity of $|X|$. We give another characterization. Write $L_p = (p-1)(p-5)/96$ for a prime $p\eq1 \pmod{4}$.

\begin{theorem}
Let $p \eq 1 \pmod{4}$ be a prime. Fix any primitive root $g$ of $p$. Then
\begin{equation}
 \prod_{1\leq i<j\leq(p-1)/2 \atop i^2+j^2 \neq 0 \pmod{p} } (j^2- i^2) \eq  (-1)^{|X|+L_p}\prod_{r,s=1 \atop r\neq s }^{(p-1)/4} (g^{4r}-g^{4s}) \pmod{p}.
 \end{equation}
\end{theorem}
\begin{proof}
Each two numbers $1\leq i<j \leq (p-1)/2$ with $i^2+j^2 \neq 0 \pmod{p}$ must belong to two different ordered pairs in our paring.
We list all the possibilities in the table. For example when $ i=a,j=\B{b}$ for $(a,b)\in X$, the term in the second row below $(a,\B{b})$
 reads $ j^2-i^2 \eq -b^2-a^2 \pmod{p}.$ \\

\renewcommand\arraystretch{1.2}
\begin{tabular}{|c|c|c|c|c|c|}\hline
\ X:\ $a< b < \B{a} < \B{b}$ & $(i,j)$ & $(a,b)$ & $(a,\B{b})$ & $(b,\B{a})$ & $ (\B{a},\B{b})$  \\ \hline
\  $\pmod{p}$ & $j^2-i^2 $ & $ b^2-a^2$ & $ -b^2-a^2$ & $ -a^2-b^2$ & $  -b^2+a^2$  \\ \hline
\ Y:\ $a< \B{a} < b < \B{b}$ & $(i,j)$ & $(a,b)$ & $(a,\B{b})$ & $(\B{a},b)$ & $ (\B{a},\B{b})$  \\ \hline
\  $\pmod{p}$  & $j^2-i^2 $ & $b^2-a^2$ & $ -b^2-a^2$ & $ b^2+a^2 $ & $ -b^2+a^2$  \\ \hline
\ Z:\ $a< b < \B{b} <\B{a} $ &  $(i,j)$ & $ (a,b)$ & $(a,\B{b})$ & $(b,\B{a})$ & $ (\B{b},\B{a})$\\ \hline
\  $\pmod{p}$     & $j^2-i^2$ & $ b^2- a^2$ & $ -b^2-a^2$ & $ -a^2-b^2$ & $ -a^2+b^2$  \\ \hline
\end{tabular} \\

Multiply the terms from each even row. It's clear that any two ordered pairs $<a,\B{a}> , <b,\B{b}>$ with $a < b$ contribute to the product
on the left of (2.2) by a factor equals :
 $$\begin{cases} \  -(b^2-a^2)^2(b^2+a^2)^2   &\t{when}\ (a,b)\in X  ,
 \\ \  (b^2-a^2)^2(b^2+a^2)^2 &\t {when}\  (a,b) \in Y \cup Z.
 \end{cases} \\$$

Thus
$$ \prod_{1\leq i<j\leq(p-1)/2 \atop i^2+j^2 \neq 0 \pmod{p} } (j^2- i^2) \eq (-1)^{|X|}\prod_{ a\neq b \in V } (a^4-b^4)^2 \pmod{p}. $$

 Moreover for any different $a,b \in V $, we have $ b^2 \neq \pm a^2 \pmod{p}$ so that $ a^4 \neq b^4 \pmod{p}.$ Thus both
 $\{ a^4 | a\in V \}$ and $\{ g^{4r}| 1\leq r \leq (p-1)/4 \}$ are the same set of quartic residues modulo $p$. Therefore
 $$ \prod_{ a\neq b \in V } (a^4-b^4)^2 \eq \prod_{1\leq r < s \leq (p-1)/4} (g^{4r}-g^{4s})^2 \pmod{p}. $$

 It's also clear that
 $$ \displaystyle \prod_{1\leq r < s \leq (p-1)/4} (g^{4r}-g^{4s})^2  =
    (-1)^{\binom{(p-1)/4}{2}}\prod_{r,s=1 \atop r\neq s }^{(p-1)/4} (g^{4r}-g^{4s})$$.

Note that $ {\binom{(p-1)/4}{2}}= 3L_p .$ The proof is complete.  \\
\end{proof}

About the left side of (2.1) we have the following relation.
\begin{equation}
\prod_{1\leq<i<j\leq(p-1)/2 \atop i^2+j^2\neq 0 \pmod{p}} (j^2-i^2) = \prod_{1\leq i < j \leq (p-1)/2} (j^2-i^2) \ / \prod_{1\leq<i<j\leq(p-1)/2 \atop
i^2+j^2\eq 0 \pmod{p}} (j^2-i^2).  \\                                                                                                                       \end{equation}

The first product on the right is easy to evaluate.

\setcounter{lemma}{2}
\begin{lemma}((1.5) \cite{S19}, p. 248)  Let $p>3$ be a prime. Then
$$ \prod_{1\leq i < j \leq (p-1)/2} (j^2-i^2) =
\begin{cases} -\prod_{x=1}^{(p-1)/2}x \pmod{p} &\t {if}\ p\eq 1 \pmod{4} ;
\\ 1 \pmod{p} &\t {if}\ p\eq3 \pmod{4}.  \\
\end{cases}$$
\end{lemma}

As for the second product on the right of (2.2), we observe that
\begin{equation}
 \prod_{1\leq<i<j\leq(p-1)/2 \atop i^2+j^2\eq 0 \pmod{p}} (j^2-i^2) = \prod_{a \in V} (\B{a}^2- a^2 ).
\end{equation}

This reminds us our discovery in the last section. That is, $< \B{a} - a,\T{\B{a}+a} >$ is also an ordered pair in our paring for any pair $<a,\B{a}>$.

Recall that for fixed $p\eq 1 \pmod{4}$ we define $ V = \{ x | 1\leq x \leq (p-1), \  x < \T{tx} \}. $
And for each $x\in V$ ,we define $\B{x}= \T{tx}$.   \\

\setcounter{theorem}{3}
\begin{theorem} \ \\

{\rm (i)}  For each $a\in V $, let $b= \B{a} - a$ ,then $b\in V$ and $ \B{b} = \T{\B{a}+a} $.         \\

{\rm (ii)} $\{ \B{a} -a \ | a \in V\} = V, \ \  \{ \T{\B{a}+a} \ | a\in V \} = \{1,2,\ldots,(p-1)/2\} \setminus V.\\$
\end{theorem}

\begin{proof}
For a fixed $a \in V$. By definition $\B{a} = \pm ta\mod {p}.$ If $\B{a}\eq ta \pmod{p}$, by definition we have $a< \B{a}=\{ta\}_p < p/2 $ so that $b = \{(t-1)a\}_p < \{ ta\}_p$.

 First we show $b= \B{a}-a \in V$. We need to check $b< \T{tb}$, i.e. $ \{(t-1)a\}_p < \T{t(t-1)a}=\T{(t+1)a}. $

  If $\{(t+1)a\}_p >p/2$ , then $\{(t-1)a\}_p < \{ta\}_p < \{(t+1)/p\}$ so that $ \{(t+1)a\}_p + \{(t-1)a\}_p = 2\{ta\}_p < p$. Thus $\T{(t+1)a} = p - \{(t+1)a\}_p > \{(t-1)a\}_p.$  Otherwise If $\{(t+1)a\}_p < p/2$ , then $\T{(t+1)a} = \{(t+1)a\}_p > \{(t-1)a\}_p$ since $ 2\{a\}_p < p $.

Next we prove $\B{b}= \T{a+\B{a}}$. Since $\B{a}\eq ta \M{p}$, multiply $ \{ta\}_p -a = b $ by t we have
$ -a - ta \eq -a - \B{a} \eq tb \M{p}$. So $\T{a+\B{a}} = \B{b}$.

If otherwise $\B{a}\eq -ta \M{p}$, we have $a< p - \{ta\}_p< p/2 $ so that $ p< 2\{ta\}_p$. We need to prove
$ \{ (-t-1)a\}_p < \T{t(-t-1)a} = \T{(1-t)a}.$ We can argue similarly. \\

Thus we establish (i) and that $\{ \B{a} -a | a \in V\} \subset V, \  \{ \T{\B{a}+a} \ | a\in V \} \subset \{1,2,\ldots,(p-1)/2\} \setminus V$. Then we assume $a,b \in V$ and $ \B{a}-a = \B{b}-b $, we claim $a=b$.

Let $V_1=\{ a \in V | \B{a} \eq at \pmod{p} \}$. If $a,b \in V_1 $ , then $(t-1)a \eq (t-1)b \pmod{p}$. So it's clear $a=b$. If $a,b \in V \setminus V_1$, it is also easy. Now suppose $a\in V_1$ but not b , i.e. $\B{a} \eq ta ,\  \B{b} \eq -ta \pmod{p}$, then $(t-1)a \eq (-t-1)b \pmod{p}$. Since  $(t-1)t\eq(-t-1)\pmod{p}$, we have $a\eq bt \eq -\B{b} \pmod{p}$ , which is impossible since $a,\B{b} < p/2$ by definition.

Hence $\{ \B{a} -a | a \in V\} $ must be identical to $V$. The second part of (ii) also follows since $\{1,2,\ldots,(p-1)/2\} \setminus V$ has as many elements as V.\\
\end{proof}

For any prime $p \eq 1 \pmod{4}$ , write $M_p = (p^2-1)/8$. It's clear $(-1)^{M_p} = (\frac{2}{p}).$

\setcounter{corollary}{4}
\begin{corollary}
{\rm (i)} $\sum_{a\in V} a = M_p/3 , \ \sum_{a\in V} \B{a} = 2M_p/3.$  \\

{\rm (ii)} Let $V_2=\{ a \in V | a+\B{a} > p/2 \}$. Then $ |V_2| \eq M_p \pmod{2}.$   \\

\end{corollary}

\begin{proof}
{\rm (i)} By Theorem 2.1 $ \sum_{a \in V} (\B{a} -a) = \sum_{a \in V} a$ so that $\sum_{a\in V} \B{a}=2 \sum_{a\in V} a$. On the other hand we have
$$\sum_{a \in V } a + \sum_{a \in V } \B{a} = 1+2+ \dots +(p+1)/2 = (p+1)(p-1)/8 = M_p.$$
Hence $ \sum_{a \in V} a = M_p/3 , \  \sum_{a\in V} \B{a} = 2M_p/3. $\\

{\rm (ii)} By definition we have $\B{a}+a = p- \T{\B{a}+a}$ for each $a \in V_2$ while $\B{a}+a = \T{\B{a}+a}$ for $a \in V\setminus V_2$. Thus
$$ \sum_{a \in V} (\B{a} +a) \eq \sum_{a\in V} (\T{\B{a}+a}) + p|V_2| \eq \sum_{a\in V} \B{a} +p|V_2| \pmod{2}, $$
in which the second congruence follows from Theorem 2.1 . Hence $|V_2| \eq \sum_{a \in V} a \eq M_p/3 \eq M_p \pmod{2}.$

\end{proof}

 Now we can evaluate the right side of (2.3) and hence the left side of (2.1).

\begin{corollary} With the same notations we have
$$\prod_{1\leq<i<j\leq(p-1)/2 \atop i^2+j^2\eq 0 \pmod{p}} (j^2-i^2) \eq \prod_{a \in V } (\B{a}^2 - a^2) \eq \(\frac{2}{p}\)\prod_{x=1}^{(p-1)/2}x \pmod{p}.
$$
\end{corollary}
\begin{proof}
By Theorem 2.1 and Corollary 2.2(ii) we have
$$ \prod_{a \in V}(\B{a} - a) = \prod_{a \in V} a , $$
$$\prod_{a \in V } (\B{a}+ a ) \eq (-1)^{|V_2|} \prod_{a \in V } \T{\B{a}+a} \eq (-1)^{M_p} \prod_{a \in V } \B{a}. \pmod{p} \\$$
 By definition of V we have $\prod_{a \in V } a\B{a} = \prod_{x=1}^{(p-1)/2}x$. Now put these together and the proof is complete.\\
\end{proof}

\begin{corollary}
Let $p\eq1\pmod{4}$ be a prime. Then
\begin{equation}
 \prod_{1\leq<i<j\leq(p-1)/2 \atop i^2+j^2\neq 0 \pmod{p}} (j^2-i^2) \eq  -(\frac{2}{p})  \pmod{p}.
 \end{equation}
\end{corollary}

\begin{proof}
It's clear from (2.2),(2,3), Lemma 2.3 , Corollary 2.6 .        \\
\end{proof}

Now we evaluate the right product of (2.1) $\prod_{r,s=1 \atop r\neq s }^{(p-1)/4} (g^{4r}-g^{4s}) \pmod{p}$ . It's well known that Jacobi Sums can be applied to count the number of solutions of diagonal equations on $ \Z / p \Z$. Fortunately our equation is very simple and there are formulas available in the monograph \cite{BEW}. We introduce some standard notations.

Let $p\eq 1 \pmod{4}$ be a prime. Let g be a primitive root of p and choose $\chi$ to be the quartic character modulo p such that $\chi(g)=i$. Let $\alpha_4$ and $\beta_4$ be the integers uniquely determined by
\begin{equation}
 \alpha_4^2+\beta_4^2=p, \ \ \alpha_4 \eq -\big(\frac{2}{p}\big)\pmod{4} , \ \beta_4\eq \alpha_4g^{(p-1)/4} \pmod{p}.
\end{equation}

For future use we list the following table. It shows that, for example, if $p \eq 1\pmod{32}$ , we have either
$ \alpha_4 \eq 7 \pmod{16}, \beta_4 \eq 4 \pmod{8}$ or $ \alpha_4 \eq 15 \pmod{16}, \beta_4 \eq 0 \pmod{8}$. \\

\begin{tabular}{|c|c|c|c|c|c|c|c|}\hline
 $p\pmod{32}$ & $\alpha_4 (16)$ & $\beta_4 (8)$ & $\alpha_4 (16)$ & $\beta_4 (8)$& $p\pmod{32}$ & $\alpha_4 (16)$ & $\beta_4 (4)$ \\ \hline
1 &  7  & 4  & 15 & 0 & 5 & 1  & 2  \\   \hline
9 &  3  & 0 & 11 & 4 & 13 & 13 & 2\\   \hline
17 &  7  & 0 &  15 & 4 & 21 & 9 & 2 \\  \hline
25   & 3  &4  & 11 &  0 & 29 & 5 & 2  \\   \hline
\end{tabular}

\setcounter{lemma}{7}
\begin{lemma}
For $m \in \Z $ prime to $p$, let $N(m)$ denotes the number of solutions of the equation of $x^4-y^4 \eq m \pmod{p}$ where $0 \leq x,y \leq (p-1)$.

When $p\eq 1 \pmod{8},$ \ $N(m) =  \begin{cases}
  p-3+6\alpha_4  &\t{when}\ \chi(m)=1 ,
\\  p-3-2\alpha_4 &\t{when}\ \chi(m)=-1  ,
\\  p-3-2\alpha_4+4\beta_4 &\t{when}\ \chi(m)=i   ,
\\  p-3-2\alpha_4-4\beta_4  &\t {when}\ \chi(m)=-i.
\end{cases} \\ \\ $

When $p\eq 5 \pmod{8},$ \ $N(m) =
\begin{cases} p-3+2\alpha_4  &\t{when}\ \big( \frac{m}{p} \big)=1 ,
\\  p-3-2\alpha_4 &\t{when}\ \big( \frac{m}{p} \big)=-1 .
\end{cases} \\   $
\end{lemma}
\begin{proof}
This follows from the formulas for $ c_1 x^4+ c_2 x^4\eq m \pmod{p}$ with general coefficients $c_1,c_2 \in \Z$ .(Corollary 10.7.2\cite{BEW}, p. 317)
\end{proof}

Since the above formulas count also solutions such as $(x,0),(0,y)$. We need do little modification.

\setcounter{corollary}{8}
\begin{corollary}
For $m \in \Z $ prime to $p$, let $N'(m)$ denotes the number of solutions of $ g^{4r}-g^{4s} \eq m\pmod{p}$ , where $ 1 \leq r,s\leq (p-1)/4$. \\

When $p\eq 1 \pmod{8},$ \ $N'(m) =  \begin{cases}
  [(p-3+6\alpha_4)-8]/16   &\t{when}\ \chi(m)=1 ,
\\  (p-3-2\alpha_4)/16 &\t{when}\ \chi(m)=-1  ,
\\  (p-3-2\alpha_4+4\beta_4)/16 &\t{when}\ \chi(m)=i   ,
\\  (p-3-2\alpha_4-4\beta_4)/16  &\t{when}\ \chi(m)=-i.
\end{cases} \\ \\ $

When $p\eq 5 \pmod{8},$ \ $N'(m) =
\begin{cases} [(p-3+2\alpha_4)-4]/16   &\t{when}\ \big( \frac{m}{p} \big)=1 ,
\\  (p-3-2\alpha_4)/16  &\t{when}\ \big( \frac{m}{p} \big)=-1 .
\end{cases} \\$
\end{corollary}
\begin{proof}
First we must exclude solutions of $x^4-y^4 \eq m \pmod{p}$ with $xy\eq 0 \pmod{p}$. When $p\eq 1 \pmod{8}$ and $\chi(m)=1$ ,$x^4\eq m \pmod{p}$ has four
solutions and so does $-y^4 \eq m \pmod{p}$ . When $p\eq 1 \pmod{8}$ and $\big( \frac{m}{p} \big)=1$, only the former one has four solutions.

Next it's clear each pair $ 1 \leq r,s\leq (p-1)/4$ corresponds to $16$ pairs of $(x,y)$ with $x^4\eq g^r, y^4 \eq g^s \pmod{p}.$
The conclusion easily follows.
\end{proof}

\setcounter{theorem}{9}
\begin{theorem} Let $p\eq 1\pmod{4}$ be a prime and $g$ be any primitive root of $p$. Write $M_p = (p^2-1)/8$. Then
\begin{equation}
\prod_{r,s=1 \atop r\neq s }^{(p-1)/4} (g^{4r}-g^{4s}) \eq  -(\frac{2}{p}) \pmod{p}.
\end{equation}
\end{theorem}

\begin{proof}
We write
$$ \prod_{r,s=1 \atop r\neq s }^{(p-1)/4} (g^{4r}-g^{4s}) \eq  \prod_{n=1}^{p-1} g^{nN'(g^n)} \pmod{p}, $$
and show the product on the right is congruent to $(\frac{2}{p})$ modulo $p$.

First assume $ p = 5+ 8k.$ Then from Corollary 2.4
\begin{align*}
 \sum_{n=1}^{p-1} nN'(g^n) &=[ 1+3+\ldots+(3+8k)](p-3-2\alpha_4)/16  \\
& \  \  \  \ + [2+4+\ldots +(4+8k)](p-3+2\alpha_4-4) /16 \\
& =[(p-3)(8k+5)(4k+2)+ 2\alpha_4(4k+2) - 4(4k+3)(4k+2) ]/16 \\
& =(4k+2)[ (p-3)(p) +2\alpha_4 - 4(4k+3) ]/16   \\
& =[(p-1)/2] [(p-3)p + 2\alpha_4 - 2(p+1)]16 \\
& =[(p-1)/2] [(p-5)p+2\alpha_4-2]/16
\end{align*}

Then from the table we easily check that $[(p-5)p+2\alpha_4-2]/16$ is even so that
$$ \prod_{n=1}^{p-1} g^{nN'(g^n)} = g^{\sum_{n=1}^{p-1} nN'(g^n)} \eq (-1)^{[(p-5)p+2\alpha_4-2]/16} \eq 1 \pmod{p}.$$

Next assume $ p= 1+8k$. Then from Corollary 2.4 again
\begin{align*}
 \sum_{n=1}^{p-1} nN'(g^n) &= \{ (1+2+\ldots+8k)(p-3-2\alpha_4)+(4+8+\ldots+8k)(8\alpha_4 -8) \\
&  \  \  \   \ +[1-3+5-7+\ldots+(8k-3)-(8k-1)](4\beta_4)\}/16 \\
& = [4kp (p-3-2\alpha_4)+ 4k(2k+1)(8\alpha_4-8)-16k\beta_4 ]/16 \\
& = 4k [ p(p-3-2\alpha_4) + (p+3)(2\alpha_4-3)-4\beta_4]/16 \\
& = [(p-1)/2] [p^2-5p-6+6\alpha_4-4\beta_4]/16
 \end{align*}
Checking every case from the table after (2.5) we have $[p^2-5p-6+6\alpha_4-4\beta_4]/16$ is always odd so
that $ \prod_{n=1}^{p-1} g^{nN'(g^n)} \eq-1 \pmod{p}$ for $ p \eq 1\pmod{8}$.

In summary $ \prod_{n=1}^{p-1} g^{nN'(g^n)} \eq -(\frac{2}{p}) \pmod{p}.$ The proof is complete.

\end{proof}

Thus from Corollary 2.7 and Theorem 2.10 that the two products in the (2.1) are congruent modulo $p$ so that
$$s'(p)\eq |X| \eq L(p) \pmod{2}.$$

Recall that $L_p = (p-1)(p-5)/96$ for a prime $p\eq1 \pmod{4}$.It's clear $L_p$ is even if $p\eq 1,5\pmod{16}$ and odd if $p\eq 9,13 \pmod{p}.$
Therefore $ L_p \eq \lfloor \frac{p+1}{8} \rfloor  \pmod{2}$.

We will discuss the case for $p\eq 3 \pmod{4}$ in Section 4 . Thus finally we obtain the following theorem.

\begin{theorem}
Let p be an odd prime and $m\in Z $ with $p\nmid m.$ Let $s'(p)$ be the number of pairs $(i,j)$ such that $1\leq i < j \leq (p-1)/2$ and  $\T{ mi^2} >
\T{mj^2}.$ Then $s'(p) \eq \lfloor \frac{p+1}{8} \rfloor  \pmod{2}$, where $ \lfloor \ \rfloor $ is the floor function.
 \end{theorem}

\section{ Balanced Pairing Theorem }

 Since the proof given in last section involves the application of Jacobi Sums of quartic character,  we try to find a more elementary proof.
 From the results we have obtained ,it's not hard to see that $ |X|\eq|Y|\eq|Z|\eq\pmod{2}$. However it's most amazing to discover that these three
 cases occur with equal frequencies, i.e. $ |X| = |Y| = |Y|$, which enables us to give a combinatoric proof.

 We know that there are $(p-1)/4$ ordered pairs so that $|X|+|Y|+|Z|= \binom{(p-1)/4}{2} = 3L_p.$ We need two more identities. The first one is found by inspecting their relative orders more closely.
\setcounter{equation}{0}
\setcounter{theorem}{0}
\begin{theorem}
\begin{equation}
|Y|= L_p , \ \  |X|+ |Z|= 2L_p.
\end{equation}
\end{theorem}

\begin{proof}
For each $a\in V$ ,define $ X_{a} = \{ x\in V | a < x , (a,x)\in X \}$ and $ X_{a}' = \{ x\in V | x < a ,  (x,a)\in X \}$. These two sets are related
to those ordered pairs in X whose first or second component is $a$ respectively. Then we define $Y_{a}$,$Y_{a}'$ and $Z_{a}$,$Z_{a}'$ similarly.
For any pair $(a,b)\in X $ with $a<b$ we have $b\in X_{a}$ and $a \in X_{b}'$. and the same holds for X replaced by Y or Z. So we have
\begin{gather}
\sum_{a\in V} |X_{a}| = \sum_{a\in V} |X_{a}'| =|X|, \notag \\
\sum_{a\in V} |Y_{a}| = \sum_{a\in V} |Y_{a}'| =|Y|, \\
\sum_{a\in V} |Z_{a}| = \sum_{a\in V} |Z_{a}'| =|Z|. \notag
\end{gather}

Now fix any $a\in V$, consider the set $\{1,2,\ldots,a-1\}.$ For any element $x\in V$ with $x<a$, as the figure shows ,
there are three possibilities before us.
$$ \ldots \square\ldots\bigcirc \ldots \triangle \ldots \B{\bigcirc} \ldots a \ldots \B{\square}  \ldots \B{a} \ldots \B{\triangle} \ldots   $$

First case, its partner $\B{x}$ is also less than $a$. (Denote x by $\bigcirc$.) Then $x < \B{x} < a <\B{a}$ so that $(x,a) \in Y$ and $ x\in Y_{a}'.$ Second
($\square$), if $ a< \B{x} <\B{a}$, then $ x<a<\B{x} <\B{a}$ so that $ (x,a) \in X$ and $x \in X{a}'$. Third ($\triangle$), if $\B{a} < \B{x}$, then
 $ x<a<\B{a} <\B{x}$ so that $ (x,a)\in Z$ and $x \in Z_{a}'$. Besides if any number less than $a$ isn't in V, say $\B{y}$, then
 its partner y is in the first case. \\

In summary for each $a\in V$ we have
$$ a-1= 2|Y_{a}'|+ |X_{a}'|+ |Z_{a}'| $$
And we add up these identities together, then by (3.1)
\begin{equation}
\sum_{a\in V} (a-1) = 2|Y|+|X|+|Z|. \notag
\end{equation}

$$ \ldots \square\ldots a \ldots \bigcirc \ldots \triangle \ldots \B{\bigcirc}\ldots \B{\square}\ldots \B{a} \ldots \B{\triangle} \ldots   $$

Then for each $a\in V$, we analyse the set $\{a+1,a+2,\ldots,\B{a}-1\}$ in the same way. It divided into three subsets.
 First ($\square$) consists of those $\B{x}$ with $x \in X_{a}'$ so that $x <a<\B{x} <\B{a}$ . Second ($\bigcirc$) consists those pairs $<x,\B{x}>$ with $x \in Z_{a} $ so that $a< x<\B{x} <\B{a}$. Then the third ($\triangle$) consists of those $x\in X_{a}$ such that $a< x<\B{a} <\B{x}.$ For each $a\in V$ we obtain  \\
 $$ \B{a}-a-1 = |X_{a}'|+ 2|Z_{a}|+ |X_{a}|.  $$

 And we add up these identities together , then by (3.7) again
\begin{equation}
 \sum_{a\in V} (\B{a}-a-1) = 2|X|+2|Z|.  \notag
\end{equation}

The left sides of (3.2) and (3.3) are equal since by Theorem 2.4.
$$ \sum_{a\in V} (\B{a}-a-1) - \sum_{a\in V} (a-1)= \sum_{a\in V} \B{a} - 2\sum_{a\in V} a = 0 .$$

Therefore $2|X|+2|Z|=2|Y|+|X|+|Z|$ so that $2|Y| = |X|+ |Z|$. We know that$|X|+|Y|+|Z|= 3L_p.$ The conclusion follows.

\end{proof}

However in the above argument $X$ and $Z$ are in the same position. So we had to turn to another way.

\begin{theorem}
Let $p \eq 1 \pmod{4}$ be a prime. Write $t = \prod_{x=1}^{(p-1)/2}x$. Let $\gamma(t,p)$ be the number of pairs $(i,j)$ such that $1\leq i< j \leq (p-1)/2 $ and $\T{ti} > \T{tj}$. Then
\begin{equation}
\gamma(p) = 4|Z|+ 2|X|+ (p-1)/4.
\end{equation}
\end{theorem}
\begin{proof}
Recall that if $a \in V$, then $\T{\B{a}}= a, \T{a}= \B{a}.$ Thus if $ i=a, j=\B{a}$ for some $a \in V$, then $\T{ti}=j > i=\T{tj}$.
There are $(p-1)/4$ such pairs. Otherwise $i,j$ belong to different ordered pairs in our paring. We list all the possibilities and mark with underline those pairs contribute to $\gamma(t,p)$. So that each $(a,b)\in X$ give two such pairs and each $(a,b)\in Z$ give four. And the lemma follows. \\

\renewcommand\arraystretch{1.4}
\begin{tabular}{|c|c|c|c|c|c|}\hline
\ X:\ $a< b < \B{a} < \B{b}$ & $(i,j)$ & $(a,b)$ & $\underline{(a,\B{b})}$ & $\underline{(b,\B{a})}$ & $ (\B{a},\B{b})$  \\ \hline
\  2    & $(\T{ti},\T{tj}) $ & $(\B{a},\B{b})$ & $(\B{a},b)$ & $(\B{b},a)$ & $ (a,b)$  \\ \hline
\ Y:\ $a< \B{a} < b < \B{b}$ & $(i,j)$ & $(a,b)$ & $(a,\B{b})$ & $(\B{a},b)$ & $ (\B{a},\B{b})$  \\ \hline
\  0     & $(\T{ti},\T{tj})$ & $(\B{a},\B{b})$ & $(\B{a},b)$ & $(a,\B{b})$ & $ (a,b)$  \\ \hline
\ Z:\ $a< b < \B{b} <\B{a} $ &  $(i,j)$ & $ \underline{(a,b)}$ & $\underline{(a,\B{b})}$ & $\underline{(b,\B{a})}$ & $ \underline{(\B{b},\B{a})}$\\ \hline
\  4     & $(\T{ti},\T{tj})$ & $(\B{a},\B{b})$ & $(\B{a},b)$ & $(\B{b},a)$ & $ (b,a)$  \\ \hline
\end{tabular} \\

\end{proof}

Then we investigate this new unknown $\gamma(t,p)$. Now note that if our theorem is valid, then $|X|=|Y|=|Z|=L_p= (p-1)(p-5)/96.$ Therefore
by (3.8), we must have $\gamma(t,p)= 6 L_p + (p-1)/4 = [(p-1)/4]^2.$ It leads us to another interesting discovery.

Recall that in number theory Gauss' lemma tell us that the Legendre symbol $(\frac{x}{p} )$ for odd prime p and $p \nmid x$ is equal to
$(-1)^{ |U_{x,p}| }$, where $U_{x,p} = \{ i | 1\leq i <j \leq (p-1)/2  ,\{ xi\}_p > p/2 \}.$

Define $\gamma(x,p)$ be the number of pairs $(i,j)$ such that $1\leq i< j \leq (p-1)/2 $ and $\T{xi} > \T{xj}$. Then $\gamma(x,p)$ and $\Gamma(x,p)$ have the following relation. It is derived by elementary arguments and also holds with p replaced by any odd nature number $n>3$.

\setcounter{lemma}{2}
\begin{lemma}
\begin{equation}
\gamma(x,p) =[(p-1)/4]^2 - [(p-1)/4 - \Gamma(x,p)]^2 .
\end{equation}
\end{lemma}

Now for each pair $<a,\B{a}>$ in our paring, we have $ \B{a} = \T{ta} \eq \pm \{ta\}_p \pmod{p}.$ Thus $ \{ t\B{a}\}_p \eq \pm \{t^2a\}_p \eq \mp \{ a\}_p \pmod{p}$ which implies $\{ta\}_p > p/2 $ if and only if $\{t\B{a}\}_p < p/2$. Hence only one of $a$ and $\B{a}$ contribute to $\Gamma(t,p)$. We have $(p-1)/4$ ordered pairs. Therefore we have

\setcounter{corollary}{3}
\begin{corollary}
\begin{equation}
\Gamma(t,p)=(p-1)/4, \  \  \gamma(t,p) = [(p-1)/4]^2.
\end{equation}
\end{corollary}

Finally we establish the elegant property from (3.6), (3.8), (3.10), which along with Theorem 2.1 also give a elementary proof of Theorem 2.11.

\setcounter{theorem}{4}
\begin{theorem}{( Balanced Pairing)} \\
Let $p \eq 1 \pmod{4}$ be a prime. Sets X,Y,Z are defined in Section 1. Then
\begin{equation}
|X| = |Y| = |Z| = (p-1)(p-5)/96.
\end{equation}
\end{theorem}

\section{Conclusion}

First we discuss Problem 1 for $p\eq 3 \pmod{4}$ as promised.

When $p\eq 3 \pmod{4}$, for any $x\in \{1,2,\ldots,(p-1)/2\}$ either $x$ or $p-x$ is a quadratic residue.Therefore for any $m$ prime to p, the mapping $\sigma_m$ sending $x$ to $\T{mx^2}$ is a permutation of the group of positive residues. And as we mentioned in the start of Section 1 this group is cyclic and generated by any primitive root $g$ of $p$ in it. Let's fix one such $g$. Moreover since $(p-1)/2$ is odd, $\T{m}$ is a square in this multiplicative group, say $\T{m}= \T{l^2}.$

Recall that the sign of composition of permutations is the product of their signs. We decompose $\sigma_m$ into four permutations of $ \{1,2,\ldots,(p-1)/2\}$. First send x to index with respect to g. Second we double the index modulo $(p-1)/2$. Then turn back to $\T{x^2}$. Finally multiply m modulo $(p-1)/2$.

Now the first and the third permutations are reverse to each other. The fourth is clear positive since it' the same to multiply $l$ twice. So the sign of $\sigma_m$ equals that of the second step, the parity of which is easy to determine.

Thus when $p\eq 3 \pmod{4}$ the problem quite easy. And in \cite{LH} it's tackled through a different decomposition of $\sigma_a$. It also follows from Theorem 1.4(i) and Corollary 1.3 in Sun's paper \cite{S19}. \\

Next we pose our own problem for further research. We have shown that for every prime $p\eq1\mod{4}$, there exists a balanced pairing of $\{1,2,\ldots,(p-1)/2\}$. Now consider $\{1,2,\ldots,n\}$ for general n, is there a balanced paring for n? For this to be true , since there will
be $(n-1)/2$ pairs and we hope there are $X,Y,Z$ each of size $\binom{(n-1)/2}{2}/3$  it must be $n\eq 0,2 \pmod{6}$. We can ask a lot questions.

\begin{problem}
(i)Find different structure of balanced pairing of set $\{1,2,\ldots,n\}$ for infinitely many $n$.\\

(ii)Investigate the conditions for $n\eq 0,2 \pmod{6}$ such that $\{1,2,\ldots,n\}$ may have a balanced paring.
\end{problem}

Thanks for your reading. Any comment is welcome.

\end{document}